\documentclass[11pt]{amsart}
\usepackage{amsmath}
\usepackage{amssymb,amscd}
\usepackage{amsthm}
\usepackage{latexsym}
\usepackage{mathtools}
\usepackage[myheadings]{fullpage}
\usepackage{color}

\usepackage{url}

\newtheorem{thm}{Theorem}[section]

\newtheorem{prop}[thm]{Proposition}
\newtheorem{cor}[thm]{Corollary}
\newtheorem{rem}[thm]{Remark}

\newtheorem{ex}[thm]{Example}

\newcommand{\m}{\mathfrak m}
\newcommand{\n}{\mathfrak n}
\newcommand{\p}{\mathfrak p}
\newcommand{\q}{\mathfrak q}

\def\opn#1#2{\def#1{\operatorname{#2}}} 
\opn\spec{Spec}
\opn\depth{depth}
\opn\height{ht}
\opn\HK{HK}

\usepackage{mleftright}
\newcommand\bino[2]{\mleft(\knds\genfrac..{0pt}{}{#1}{#2}\knds\mright)}
\newcommand{\knds}{\kern-\nulldelimiterspace}

\title{The Hilbert-Kunz function of some quadratic quotients of the Rees algebra}

\author{Francesco Strazzanti}
\address{F. Strazzanti - Dipartimento di Matematica "Giuseppe Peano" - Universit\`a degli Studi di Torino - Via Carlo Alberto 10 - 10123 Torino - Italy.}
\email{francesco.strazzanti@gmail.com}

\author{Santiago Zarzuela Armengou}
\address{S. Zarzuela Armengou - Departament de Matem\`atiques i Inform\`atica - Universitat de Barcelona - Gran Via de les Corts Catalanes 585 - 08007 Barcelona - Spain.}
\email{szarzuela@ub.edu}

\thanks{The first author was supported by INdAM, more precisely he was "titolare di una borsa per l'estero dell'Istituto Nazionale di Alta Matematica" and "titolare di un Assegno di Ricerca dell'Istituto Nazionale di Alta Matematica". The second author was supported by Spanish Ministerio de Ciencia e Innovaci\'on Project PID2019-104844GB-100}


\begin{document}

\begin{abstract}
 Given a commutative local ring $(R,\mathfrak m)$ and an ideal $I$ of $R$, a family of quotients of the Rees algebra $R[It]$ has been recently studied as a unified approach to the Nagata's idealization and the amalgamated duplication and as a way to construct interesting examples, especially integral domains. When $R$ is noetherian of prime characteristic, we compute the Hilbert-Kunz function of the members of this family and, provided that either $I$ is $\mathfrak{m}$-primary or $R$ is regular and F-finite, we also find their Hilbert-Kunz multiplicity. Some consequences and examples are explored.
\end{abstract}


\maketitle

\section{Introduction}

The Hilbert-Kunz function and the Hilbert-Kunz multiplicity were introduced by Monsky in \cite{M}, even though they were already present in the work of Kunz \cite{K,K2}. For a noetherian commutative ring of prime characteristic, they provide a theory very similar to the Hilbert-Samuel theory in some respects but very different in others.
For this reason, in the last decades these notions have been explored by many researchers, becoming probably one of the main topic in the study of noetherian rings of prime characteristic, see e.g. Huneke's survey \cite{H}. To be more precise, let $(R,\m)$ be a $d$-dimensional local noetherian commutative ring with identity of prime characteristic $p$.
If $\m=(f_1, \dots, f_r)$ and $e$ is a non-negative integer, we set
$\m^{[p^e]}=(f^{p^e} \mid f \in \m)=(f_1^{p^e}, \dots, f_r^{p^e})$.
The Hilbert-Kunz function of $R$ is given by the map $\mathbb{N} \rightarrow \mathbb{N}$, $e \mapsto \ell_{R}(R/\m^{[p^e]})$, whereas the Hilbert-Kunz multiplicity of $R$ is defined as
\[
e_{\HK}(R):= \lim_{e \rightarrow \infty} \frac{\ell_{R}(R/\m^{[p^e]})}{p^{ed}}.
\]
Monsky \cite{M} proved that this limit always exists. This is a positive real number, but it can be also irrational, see \cite{Br}.
One of the most interesting results related to the Hilbert-Kunz multiplicity is that $R$ is regular if and only if $e_{\HK}(R)=1$, provided that $R$ is unmixed \cite{WY}. Hence, $e_{\HK}(R)$ can be seen as a measure of how far $R$ is from being regular.

Given an ideal $I$ of $R$, the Rees algebra of $R$ with respect to $I$ is defined as  $R[It]= \oplus_{i \geq 0} I^i t^i \subseteq R[t]$, where $t$ is an indeterminate. Consider the polynomial $t^2+at+b$ with $a, b \in R$ and let $(I^2(t^2+at+b))$ denote the contraction to $R[It]$ of the ideal generated by $t^2+at+b$ in $R[t]$. In \cite{BDS} the family of rings $R(I)_{a,b}:=R[It]/(I^2(t^2+at+b))$ is studied. It can be considered as a unified approach to some well-studied constructions: $R(I)_{0,0}$ is isomorphic to the Nagata's idealization $R \ltimes I$ \cite{AW,N}, $R(I)_{-1,0}$ is isomorphic to the amalgamated duplication $R \Join I$ of $R$ along $I$ \cite{D,DF}, $R(I)_{0,b}$ is isomorphic to the pseudocanonical double cover of $R$ \cite{E,E2} provided that $I$ is a canonical ideal of $R$. Recently the properties of this family have been investigated by many researchers, e.g. in \cite{BDS2,B,DJS,DS,F,S,T}. Also, in \cite{OST} the family is used to construct one-dimensional Gorenstein local rings having decreasing Hilbert function. The aim of this paper is to compute the Hilbert-Kunz function and the Hilbert-Kunz multiplicity of the rings $R(I)_{a,b}$. In particular, if $(R,\m)$ is a local noetherian ring, in Theorem \ref{main} we prove the following formula for the Hilbert-Kunz function of $R(I)_{a,b}$:
\begin{equation*}
\ell_{R(I)_{a,b}}\left(\frac{R(I)_{a,b}}{\n^{[q]}}\right)=
\ell_R\left(\frac{R}{\m^{[q]}}\right)+\ell_R\left(\frac{I}{\m^{[q]}I}\right) -\ell_R\left(\frac{\m^{[q]}I+ B_q I^{[q]}}{\m^{[q]}I}\right),
\end{equation*}
where $q=p^e$, $\n$ is the maximal ideal of the local ring $R(I)_{a,b}$ and $B_q$ is an element of $R$ depending only on $a$, $b$, and $q$. We also provide many cases in which $B_q$ is always invertible (Remark \ref{Particular values}) or always zero (Corollary \ref{Equality holds}) for every $q$. Moreover, if either $I$ is an $\m$-primary ideal and $\dim(R)>0$ or $R$ is regular and F-finite, we prove that $e_{\HK}(R(I)_{a,b})=2e_{\HK}(R)$, see Theorem \ref{main} and Corollary \ref{regular ring} respectively. We also include some examples showing that the existence of the second coefficient of Hilbert-Kunz function of $R$ (resp. $R(I)_{a,b}$) does not imply the existence of the same coefficient for $R(I)_{a,b}$ (resp. $R$). Finally we give a method to produce infinitely many integral domains whose second coefficients of the Hilbert-Kunz function is periodic.

\section{The Hilbert-Kunz function of $R(I)_{a,b}$}

Let $R$ be a commutative ring, let $I$ be an ideal of $R$ and let $a$, $b \in R$. In Introduction we have defined the rings $R(I)_{a,b}$.
Since $R(I)_{a,b} \cong R \oplus I$ as $R$-modules, we denote the elements of $R(I)_{a,b}$ by $r+it$, where $r \in R$ and $i \in I$, and in this case we refer to $r$ and $it$ as the first and the second term of $r+it$ respectively.
With this notation, the multiplication in $R(I)_{a,b}$ is given by $(r+it)(s+jt)=rs-bij+(rj+si-aij)t$. We often assume that $I \neq R$, but notice that $R(R)_{a,b} \cong R[t]/(t^2+at+b)$ as rings.

If $(R,\m,k)$ is a noetherian local ring with dimension $d$, also $R(I)_{a,b}$ is a $d$-dimensional noetherian local ring with maximal ideal $\n=\{m+it \mid m \in \m, i \in I\}$ and $R(I)_{a,b}/\n \cong k$, see \cite{BDS}. Moreover, there is a natural injective homomorphism of rings $R \rightarrow R(I)_{a,b}$ and, since the two rings have the same residue field, this implies that every $R(I)_{a,b}$-module $M$ is an $R$-module and $\ell_R(M)=\ell_{R(I)_{a,b}}(M)$, where $\ell_S(\cdot)$ denotes the length as module over a ring $S$.

In this paper we denote the set of non-negative integers by $\mathbb{N}$.

\begin{prop} \label{Formula for the powers}
Let $R$ be a commutative ring, let $I$ be an ideal of $R$ and let $a, b \in R$. Then, for every $r+it \in R(I)_{a,b}$ and every $n \in \mathbb{N}$, it holds
\begin{align*}
(r+it)^n= &\left( r^n + \sum_{j=2}^n \binom{n}{j} \sum_{\substack{u,v \in \mathbb{N} \\ u+2v=j-2}} \binom{u+v}{u} \, (-a)^u \, (-b)^{v+1} \, i^{j} \, r^{n-j} \right) + \\
&\left( \sum_{j=1}^n \binom{n}{j} \sum_{\substack{u,v \in \mathbb{N} \\ u+2v=j-1}} \binom{u+v}{u} \, (-a)^u \, (-b)^v \, i^{j} \, r^{n-j} \right)t.
\end{align*}
\end{prop}

\begin{proof}
We prove it by induction on $n$. The base case is trivial, so we assume that the formula is true for $(r+it)^n$ and consider the first term of $(r+it)^{n+1}=(r+it)^{n}(r+it)$, which is
\begin{gather*}
r^{n+1} + \sum_{j=2}^n \binom{n}{j} \sum_{\substack{u,v \in \mathbb{N} \\ u+2v=j-2}} \binom{u+v}{u} \, (-a)^u \, (-b)^{v+1} \, i^{j} \, r^{n+1-j} \ + \\
+ \sum_{j=1}^n \binom{n}{j} \sum_{\substack{u,v \in \mathbb{N} \\ u+2v=j-1}} \binom{u+v}{u} \, (-a)^u \, (-b)^{v+1} \, i^{j+1} \, r^{n-j} = \\
=r^{n+1} + \sum_{j=2}^n \binom{n}{j} \sum_{\substack{u,v \in \mathbb{N} \\ u+2v=j-2}} \binom{u+v}{u} \, (-a)^u \, (-b)^{v+1} \, i^{j} \, r^{n+1-j} \ + \\
+ \sum_{j=2}^{n+1} \binom{n}{j-1} \sum_{\substack{u,v \in \mathbb{N} \\ u+2v=j-2}} \binom{u+v}{u} \, (-a)^u \, (-b)^{v+1} \, i^{j} \, r^{n+1-j} = \\
=r^{n+1} + \sum_{j=2}^{n+1} \binom{n+1}{j} \sum_{\substack{u,v \in \mathbb{N} \\ u+2v=j-2}} \binom{u+v}{u} \, (-a)^u \, (-b)^{v+1} \, i^{j} \, r^{n+1-j}.
\end{gather*}
On the other hand, the coefficient $X$ of the second term is equal to
\begin{gather*}
X:= \sum_{j=1}^n \binom{n}{j} \sum_{\substack{u,v \in \mathbb{N} \\ u+2v=j-1}} \binom{u+v}{u} \, (-a)^u \, (-b)^v \, i^{j} \, r^{n+1-j} \ + \\
+ i  r^n  + \sum_{j=2}^n \binom{n}{j} \sum_{\substack{u,v \in \mathbb{N} \\ u+2v=j-2}} \binom{u+v}{u} \, (-a)^u \, (-b)^{v+1} \, i^{j+1} \, r^{n-j}  \ + \\
 + \sum_{j=1}^n \binom{n}{j} \sum_{\substack{u,v \in \mathbb{N} \\ u+2v=j-1}} \binom{u+v}{u} \, (-a)^{u+1} \, (-b)^v \, i^{j+1} \, r^{n-j} = \\
 = n i r^n - \binom{n}{2} a i^2 r^{n-1} + \sum_{j=3}^{n} \binom{n}{j} \sum_{\substack{u,v \in \mathbb{N} \\ u+2v=j-1}} \binom{u+v}{u} \, (-a)^u \, (-b)^v \, i^{j} \, r^{n+1-j}  + i r^n - n a i^2 r^{n-1}  +  \\
+ \sum_{j=2}^{n} \binom{n}{j} \left(\sum_{\substack{u,v \in \mathbb{N} \\ u+2v=j-2}} \binom{u+v}{u} \, (-a)^u \, (-b)^{v+1} +  \sum_{\substack{u,v \in \mathbb{N} \\ u+2v=j-1}} \binom{u+v}{u} \, (-a)^{u+1} \, (-b)^v \right) i^{j+1} \, r^{n-j}.
\end{gather*}
If $j \geq 2$, we note that
\begin{gather*}
\sum_{\substack{u,v \in \mathbb{N} \\ u+2v=j-2}} \binom{u+v}{u} \, (-a)^u \, (-b)^{v+1} +  \sum_{\substack{u,v \in \mathbb{N} \\ u+2v=j-1}} \binom{u+v}{u} \, (-a)^{u+1} \, (-b)^v = \\
=\sum_{\substack{u,v \in \mathbb{N}, v \neq 0 \\ u+2v=j}} \binom{u+v-1}{u} \, (-a)^u \, (-b)^{v} +  \sum_{\substack{u,v \in \mathbb{N}, u \neq 0  \\ u+2v=j}} \binom{u-1+v}{u-1} \, (-a)^{u} \, (-b)^v = \\
 = \sum_{\substack{u,v \in \mathbb{N} \\ u+2v=j}} \binom{u+v}{u} \, (-a)^u \, (-b)^{v}.
\end{gather*}
Therefore, it follows that
\begin{gather*}
 X = (n+1) i  r^n - \binom{n+1}{2} a  i^2  r^{n-1} +  \sum_{j=3}^{n} \binom{n}{j} \sum_{\substack{u,v \in \mathbb{N} \\ u+2v=j-1}} \binom{u+v}{u} \, (-a)^u \, (-b)^v \, i^{j} \, r^{n+1-j} \ +  \\
+ \sum_{j=3}^{n+1} \binom{n}{j-1} \sum_{\substack{u,v \in \mathbb{N} \\ u+2v=j-1}} \binom{u+v}{u} \, (-a)^u \, (-b)^{v} i^{j} \, r^{n+1-j}= \\
 = \sum_{j=1}^{n+1} \binom{n+1}{j} \sum_{\substack{u,v \in \mathbb{N} \\ u+2v=j-1}} \binom{u+v}{u} \, (-a)^u \, (-b)^v \, i^{j} \, r^{n+1-j}. \qedhere
\end{gather*}
\end{proof}

It will be useful to specialize the formula in the previous proposition for some particular members of the family $R(I)_{a,b}$.

\begin{cor} \label{Fibonacci}
Let $R$ be a commutative ring and let $I$ be an ideal of $R$. Let $r \in R$, $i \in I$ and $n \in \mathbb{N}$.
\begin{enumerate}
\item In $R(I)_{-1,0} \cong R \Join I$ we have $(r+it)^n=r^n + \sum_{j=1}^n \bino{n}{j} \, i^j \, r^{n-j} \,t$.
\item In $R(I)_{0,b}$ we have
\[
(r+it)^n= \sum_{j=0}^{\left\lfloor \frac{n}{2}\right\rfloor} \binom{n}{2j} \, (-b)^j \, i^{2j} \, r^{n-2j} + \left(\sum_{j=0}^{\left\lfloor \frac{n-1}{2}\right\rfloor} \binom{n}{2j+1} \, (-b)^j \, i^{2j+1} \, r^{n-2j-1}\right)t.
\]
\item In $R(I)_{-1,-1}$ we have $(r+it)^n=r^n + \sum_{j=2}^n \bino{n}{j} \, F_{j-1} \, i^j \, r^{n-j} + \sum_{j=1}^n \bino{n}{j} \, F_j \, i^j \, r^{n-j} \, t$, where $F_k$ denotes the $k$-th Fibonacci number.
\end{enumerate}
\end{cor}

\begin{proof}
Using the previous lemma, the first two formulas follow by an easy calculation. As for the last one, we note that $F_j=\sum_{v=0}^{\lfloor \frac{j-1}{2} \rfloor} \bino{j-1-v}{v} = \sum_{\substack{u,v \in \mathbb{N} \\ u+2v=j-1}}\bino{u+2v-v}{v}=\sum_{\substack{u,v \in \mathbb{N} \\ u+2v=j-1}}\bino{u+v}{u}$.
\end{proof}

Now we are ready to prove a formula for the Hilbert-Kunz function of $R(I)_{a,b}$. From now on $R$ will be a commutative noetherian ring with identity.

\begin{thm} \label{main}
Let $(R, \m)$ be a local ring of prime characteristic $p$, let $a, b \in R$ and let $I \neq R$ be a non-zero ideal of $R$.
\begin{enumerate}
\item For every $q=p^e$ with $e \in \mathbb{N}$, it holds $(r+it)^{q}=r^{q}+ A_q i^{q} +(B_q i^q)t$,
where
\[ A_q= \sum_{\substack{u,v \in \mathbb{N} \\ u+2v=q-2}} \binom{u+v}{u} \, (-a)^u \, (-b)^{v+1}, \hspace{30pt} B_q=\sum_{\substack{u,v \in \mathbb{N} \\ u+2v=q-1}} \binom{u+v}{u} \, (-a)^u \, (-b)^v
\]
are elements of $R$ and depend only on $a$, $b$, and $q$. In particular, if $R$ is $F$-finite, then also $R(I)_{a,b}$ is $F$-finite for every $a$ and $b$.
\item For every $q=p^e$ with $e \in \mathbb{N}$, it holds
\begin{equation*}
\ell_{R(I)_{a,b}}\left(\frac{R(I)_{a,b}}{\n^{[q]}}\right)=
\ell_R\left(\frac{R}{\m^{[q]}}\right)+\ell_R\left(\frac{I}{\m^{[q]}I}\right) -\ell_R\left(\frac{\m^{[q]}I+ B_q I^{[q]}}{\m^{[q]}I}\right).
\end{equation*}
\item If $I$ is $\m$-primary and $\dim(R) > 0$, then $e_{\HK}(R(I)_{a,b})=2e_{\HK}(R)$.
\end{enumerate}
\end{thm}

\begin{proof}
(1) The formula follows from Proposition \ref{Formula for the powers}, since $p$ divides $\bino{q}{j}$ for every $j \neq 0,q$. Assume now that $R$ is $F$-finite. Since $R(I)_{a,b}$ is a finite extension of $R$, we obtain that $R(I)_{a,b}$ is a finite extension of $R^p$. As $R^p \subset (R(I)_{a,b})^p$, we get that $R(I)_{a,b}$ is $F$-finite too.   \\
(2) Every element in $\n^{[q]}$ is a finite sum of elements in the form
\[
\left(m^q + A_q i^q + B_qi^qt\right)(s+jt)
=m^q s + A_q i^q s- bB_q i^q j + (m^q j + A_q i^q j + B_q i^q s - aB_q i^q j)t
\]
for some $m \in \m$, $i,j \in I$ and $s \in R$. It follows that $\n^{[q]} \subseteq \{m + (m'i+B_q i')t  \mid m,m' \in \m^{[q]}, i \in I, i' \in I^{[q]}\}$ and it is easy to see that this is an equality. Therefore, there is an isomorphism $\n^{[q]} \cong \m^{[q]} \oplus (\m^{[q]}I + B_qI^{[q]})$ as $R$-modules and, since the length as $R(I)_{a,b}$-module and as $R$-module is the same, we get
\[\ell_{R(I)_{a,b}}\left(\frac{R(I)_{a,b}}{\n^{[q]}}\right)=
\ell_{R}\left(\frac{R \oplus I}{\m^{[q]} \oplus (\m^{[q]}I + B_qI^{[q]})}\right)=
\ell_R\left(\frac{R}{\m^{[q]}}\right)+\ell_R\left(\frac{I}{\m^{[q]}I+ B_q I^{[q]}}\right).
\]
Hence, the formula in (2) follows from the exact sequence
\[
0 \rightarrow \frac{\m^{[q]}I+ B_q I^{[q]}}{\m^{[q]}I} \rightarrow \frac{I}{\m^{[q]}I} \rightarrow \frac{I}{\m^{[q]}I+ B_q I^{[q]}} \rightarrow 0.
\]
(3) Computing the limit of the equality in (2) divided by $q^d$ for $e \rightarrow \infty$, we get
\[
e_{\HK}(R(I)_{a,b})=e_{\HK}(R)+e_{\HK}(\m,I)- \lim_{e \rightarrow\infty } \frac{\ell_R\left(\frac{\m^{[q]}I+ B_q I^{[q]}}{\m^{[q]}I}\right)}{q^d},
\]
where $e_{\HK}(\m,I)$ denotes the Hilbert-Kunz multiplicity of $I$ as an $R$-module.
Applying \cite[Corollary 3.12]{H} to the exact sequence $0 \rightarrow I \rightarrow R \rightarrow R/I \rightarrow 0$, we get $e_{\HK}(R)=e_{\HK}(\m,I)+e_{\HK}(\m,R/I)=e_{\HK}(\m,I)$, where $e_{\HK}(\m,R/I)$ is zero because $\dim(R/I) < \dim(R)$.
Moreover,
\[
0 \leq \ell_R\left(\frac{\m^{[q]}I+ B_q I^{[q]}}{\m^{[q]}I} \right) \leq \ell_R\left(\frac{\m^{[q]}I+ I^{[q]}}{\m^{[q]}I} \right)= \ell_R\left(\frac{I^{[q]}}{\m^{[q]}I \cap I^{[q]}} \right)
\]
by the second Isomorphism Theorem.
Since $I^{[q]}I \subseteq (\m^{[q]}I \cap I^{[q]})$, it is enough to prove that
\begin{equation} \label{limit}
\tag{$\star$}
 \lim_{e \rightarrow \infty } \frac{\ell_R \left(I^{[q]}/I^{[q]}I\right)}{q^{d}}=0.
\end{equation}
Since $I$ is an $\m$-primary ideal, applying \cite[Proposition 3.11]{H} to the short exact sequence $0 \rightarrow I \rightarrow R \rightarrow R/I \rightarrow 0$ and using the short exact sequence $0 \rightarrow I/I^{[q]}I \rightarrow R/I^{[q]}I \rightarrow R/I \rightarrow 0$, we get
\begin{align*}
\ell_R \left(\frac{R}{I^{[q]}} \right)&=\ell_R \left( \frac{I}{I^{[q]}I}\right)+ \ell_R \left(\frac{R/I}{I^{[q]}R/I}\right) + O(q^{d-1})= \\
&=  \ell_R\left(\frac{R}{I^{[q]}I}\right)-\ell_R\left(\frac{R}{I}\right)+ \ell_R\left(\frac{R}{I}\right)+O(q^{d-1})=\\
&=\ell_R\left(\frac{R}{I^{[q]}I}\right)+O(q^{d-1}),
\end{align*}
because $I^{[q]}R/I=0$. This equality and the short exact sequence
$0 \rightarrow I^{[q]}/I^{[q]}I \rightarrow R/I^{[q]}I \rightarrow R/I^{[q]} \rightarrow 0$
give
\[
\ell_R\left(\frac{I^{[q]}}{I^{[q]}I}\right)=\ell_R\left(\frac{R}{I^{[q]} I}\right) - \ell_R\left(\frac{R}{I^{[q]}}\right)=O(q^{d-1}),
\]
which implies Equality (\ref{limit}). An alternative way to prove Equality (\ref{limit}) is to use that $\ell_R(I^{[q]}/I^{[q]}I) \leq \mu(I^{[q]}/I^{[q]}I)\ell_R(R/I) \leq \mu(I) \ell_R(R/I)$ for all $q$, where $\mu(\cdot)$ denotes the number of minimal generators of an $R$-module.
\end{proof}

\begin{rem} \rm \label{Particular values}
It is possible to find $B_q$ for some interesting values of $a,b \in R$. For instance, we notice the following three cases:
\begin{itemize}
\item If $a=b=0$, then $R(I)_{0,0} \cong R \ltimes I$ and $B_q=0$ for every $q>1$;
\item If $a$ is invertible in $R$ and $b=0$, then $B_q$ is invertible in $R$ for every $q>1$;
\item If $a=b=-1$, then $B_q$ is equal to the $q$-th Fibonacci number $F_q$. In particular $B_q=0$ for every $q$ if $p=5$, while $B_q$ is invertible for every $q$ if $p \neq 5$.
\end{itemize}
The first two cases follows by definition of $B_q$. We also notice that the second case includes $R(I)_{-1,0} \cong R \Join I$.
As for the last one, we have already seen that $B_q=F_q$ in the proof of Corollary \ref{Fibonacci}. Therefore, when $p=5$, we get $\gcd(F_q,F_5)=F_{\gcd(q,5)}=F_5=5$ and, thus, $B_q=0$. On the other hand, if $p\neq 5$, it is well known that $p$ divides either $F_{p-1}$ or $F_{p+1}$ and in both cases $\gcd(F_q,F_{p\pm 1})=F_{\gcd(q,p \pm 1)}=F_1=1$. Hence, $F_q$ is not divided by $p$.
\end{rem}

If $(R,\m)$ is a $d$-dimensional local ring of prime characteristic $p$, we set $f_e(R)=(1/p^{ed}) \ell_R(R/\m^{[p^e]})$. We note that $\lim_{e \rightarrow \infty} f_e(R)=e_{\HK}(R)$.
Moreover, if $\mathfrak \p$ is a prime ideal of $R$, we set $f_e(\mathfrak{\p})=f_e(R_{\mathfrak{\p}})$. In \cite[Proposition 3.3]{K2} it is proved that $f_e(\mathfrak{p})\leq f_e(\mathfrak{q})$, provided that $\mathfrak{\p} \subseteq \mathfrak{\q}$ are prime ideals and $R$ is excellent and locally equidimensional.

\begin{cor} \label{regular ring}
If $R$ is a regular F-finite local ring of positive characteristic, $I \neq R$ is a non-zero ideal of $R$ and $a,b \in R$, then $e_{\HK}(R(I)_{a,b})=2$.
\end{cor}

\begin{proof}
Let $\p$ be a minimal prime of $I$ and let $\q$ be a prime ideal of $R(I)_{a,b}$ that lies over $\p$. Note that the ideal $q$ always exists because $R(I)_{a,b}$ is an integral extension of $R$ (in \cite[Proposition 1.2]{DS} it is given an explicit and complete description of the possible $\q$). Since $I \subseteq \p$, by \cite[Proposition 1.4]{DS} it follows that $(R(I)_{a,b})_{\q} \cong R_{\p}(I_{\p})_{a,b}$. Moreover, $(R_{\p},\p R_{\p})$ is a regular local ring and $I_{\p}$ is a $\p R_{\p}$-primary ideal. Thus, Theorem \ref{main} implies that $e_{\HK}(R_\p(I_\p)_{a,b})=2e_{\HK}(R_\p)=2$. Consequently, since $R(I)_{a,b}$ is F-finite by Theorem \ref{main}, and then excellent, we get
\[
2=e_{\HK}((R(I)_{a,b})_{\q})=\lim_{e \rightarrow \infty} f_e(\q) \leq \lim_{e \rightarrow \infty} f_e(\n)=e_{\HK}(R(I)_{a,b}) \leq 2,
\]
where the last inequality follows by Theorem \ref{main} (2) because $I$ has positive height and, therefore, $e_{\HK}(\m,I)=e_{\HK}(R)$. Hence, $e_{\HK}(R(I)_{a,b}) = 2$.
\end{proof}

\begin{rem} \rm
Let $A$ be a 2-dimensional Cohen-Macaulay local ring.
In \cite[Theorem 5.4]{WY} it is proved that $e_{\HK}(A)=2$ if and only if $A$ is either a non-F-rational double point or $A$ is the ``ordinary triple point'', i.e. $\widehat{A} \cong k[[x^3,x^2y,xy^2,y^3]]$. If $A$ is not Cohen-Macaulay, it is possible to use the previous corollary to show that this classification is far from being true. More precisely, it is possible to construct a 2-dimensional local ring $A$ with $e_{\HK}(A)=2$ and Hilbert-Samuel multiplicity $n$ for every integer $n>2$. Consider a 2-dimensional regular F-finite local ring $(R,\m)$  and let $I$ be an $\m$-primary ideal with $n-1$ generators (for instance, $I$ could be generated by $x^n$, $y^n$ and other $n-3$ monomials with degree $n$, where $\m=(x,y)$).
By the previous corollary we have $e_{\HK}(R(I)_{a,b})=2e_{\HK}(R)=2$ for every $a,b$. Moreover, by \cite[Proposition 2.3]{BDS} and \cite[Corollary 5.9]{DFF}, the Hilbert-Samuel multiplicity of $R(I)_{a,b}$ is $e(R(I)_{a,b})=e(R)+\ell_{R}(I/\m I)=1+n-1=n$, because $\m$ is a minimal reduction of itself.
\end{rem}

Since the Nagata's idealization is isomorphic to $R(I)_{0,0}$, it is easy to see that the last term of the formula in Theorem \ref{main}(2) is zero, as it is possible to prove directly. The next corollary collects other cases in which this occurs.

\begin{cor} \label{Equality holds}
Let $(R,\m,k)$ be a local ring of prime characteristic $p$, let $a, b \in R$ and let $I \neq R$ be a non-zero ideal of $R$. Consider the following properties:
\begin{enumerate}
\item $ I \subseteq \m^{[p]}$;
\item $R=k[[x_1, \dots, x_n]]/J$, both $I$ and $J$ are monomial ideals and $I$ is contained in $(m^2 \mid m\in \m \text{ is a monomial}\,)$;
\item $a=b=-1$ and $p=5$;
\item $a=0$ and $p=2$.
\end{enumerate}
If one of the previous properties holds, then
\begin{equation}\label{Equality on function}
\tag{a}
\ell_{R(I)_{a,b}}\left(\frac{R(I)_{a,b}}{\n^{[q]}}\right)=\ell_R\left(\frac{R}{\m^{[q]}}\right)+ \ell_R\left(\frac{I}{\m^{[q]}I}\right)
\end{equation}
for every positive integer $q$. In particular, $e_{\HK}(R(I)_{a,b})=e_{\HK}(R)+e_{\HK}(\m,I)$.
\end{cor}

\begin{proof}
By Theorem \ref{main} it is enough to show that $I^{[q]} \subseteq \m^{[q]}I$ for every integer $q=p^e$ with $e \in \mathbb{N}$ and $e \geq 1$. Moreover, we can reduce ourselves to consider only the generators $f^q$ of $I^{[q]}$, where $f$ is a generator of $I$. \\
(1) Since $ I \subseteq \m^{[p]}$, we have $f=\sum_{i=1}^n r_i m_i^p$ for some $r_i \in R$, $m_i \in \m$ and $n \in \mathbb{N}$. Therefore, we get
\[f^q=\left(\sum_{i=1}^n r_i m_i^p\right)^{(p-1)p^{e-1}}\left(\sum_{i=1}^n r_i m_i^p\right)^{p^{e-1}}=f^{(p-1)p^{e-1}} \sum_{i=1}^n r_i^{p^{e-1}} m_i^q \in I\m^{[q]}.
\]
(2) We can write $f=\sum_{i=1}^n \, r_i m_i^2$ for some $r_i \in R$ and $m_i \in \m$. Moreover, we can assume that $n=1$ because both $I$ and $J$ are monomial ideals. Thus, we have $f=rm^2$ with $r \in R$ and $m \in \m$. It follows that $f^q=(m^q)(rm^2)(r^{q-1}m^{q-2}) \in \m^qI$.\\
(3)-(4) In these cases $B_q=0$ by Remark \ref{Particular values} and by definition of $B_q$ respectively.
\end{proof}

\begin{rem} \rm
In general Equality (\ref{Equality on function}) does not hold as the following example shows. We also notice that here $I$ is contained in $(m^{p-1} \mid m \in \m)$ but not in $\m^{[p]}$. Moreover, since $p=3$, the example also shows that it is not possible to drop the monomial hypothesis on $I$ in the second condition of the previous corollary.

Let $k=\mathbb{Z}/3\mathbb{Z}$, $R=k[[x,y]]$, $I=(x^2-y^2)$, $a=-1$ and $b=0$, then $R(I)_{a,b} \cong R \Join I$ and $B_q=1$ for every $q$. We claim that in this case
$x^{2q}-y^{2q} \in I^{[q]} \setminus \m^{[q]}I$, and therefore the equality (\ref{Equality on function}) does not hold for every $q$.
Clearly, $x^{2q}-y^{2q}=(x^2-y^2)^q \in I^{[q]}$, moreover $x^{2q}-y^{2q}=(x^{2}-y^{2})(x^{2q-2}+x^{2q-4}y^{2}+ \dots + y^{2q-2})$ and, since $I$ is principal, it is enough to show that the second factor is not in $\m^{[q]}$. On the other hand, all of its addends are in $\m^{[q]}$ except $x^{q-1}y^{q-1}$ and, so, it is not in $\m^{[q]}$.
\end{rem}

Corollary \ref{Equality holds} says that, under appropriate conditions, the Hilbert-Kunz function of $R(I)_{a,b}$ is the sum of the Hilbert-Kunz functions of $R$ and $I$, where the latter is seen as an $R$-module. The previous remark shows that in general this is not true. Also, if $B_q$ is invertible for every $q$ and $I=\m$, it is easy to see that $\ell_R((\m^{[q]}I+ B_q I^{[q]})/\m^{[q]}I)$ is equal to the number of the minimal generators of $\m^{[q]}$. In general the length of this module can lead to Hilbert-Kunz functions whose behaviour is very different from the one of $R$.

\section{Examples for the second coefficient of the Hilbert-Kunz function}

If $\ell_R(R/\m^{[q]})=e_{\HK}(R) q^d + \beta q^{d-1}+ O(q^{d-1})$  for some real number $\beta$ independent of $q$, we refer to $\beta$ as the second coefficient of the Hilbert-Kunz function of $R$. It is well known that this coefficient may not exist, even though its existence has been proved for large families of rings.

\begin{ex} \rm
Consider the ring $R=k[x,y]/(x^3+y^3)$ with $k$ a field of prime characteristic $p$ and let $\m=(x,y)$ and $I=(x)$. Assume also that $p \equiv 2 \mod 3$. Then, it is not difficult to see that
\[
\ell_R\left(\frac{R}{\m^{[q]}}\right)=\ell_R\left(\frac{I}{\m^{[q]}I}\right)=3q-2.
\]
Indeed, $I$ and $R$ are isomorphic as $R$-modules and $R/\m^{[q]}\cong k[x,y]/(x^q,y^q,x^3+y^3)$. If $q \equiv m \mod 3$, then a Gr\"obner basis of $(x^q,y^q,x^3+y^3)$ with respect to the lexicographic order $x>y$ is $\{y^q,x^3-y^3,x^m y^{q-m}\}$ and, so, its initial ideal is $J=(y^q,x^3,x^m y^{q-m})$. Therefore, the non-zero monomials of $k[x,y]/J$ are $x^i y^j$ for $i=0,1,2$ and $j=0, \dots, q-1$ except $x^m y^{q-m}$ and $x^2y^{q-1}$, which are $3q-2$ many monomials.

Despite this, if $B_q$ is invertible for every $q$, the second coefficient of the Hilbert-Kunz function of $\widehat{R}(\widehat{I})_{a,b}$ does not exist. Indeed, $(\m^{[q]}I+ B_q I^{[q]})/\m^{[q]}I=(x^q,xy^q)/(x^{q+1},xy^q)$ and
\[
\ell_R\left( \frac{(x^q,xy^q)}{(x^{q+1},xy^q)}\right)=
\begin{cases}
2 \ \ \ \text{ if } e \text{ is odd} \\
1 \ \ \ \text{ if } e \text{ is even}
\end{cases}
\]
because $x^qy$ is in $(x^{q+1},xy^q)$ if and only if $q \equiv 1 \mod 3$, i.e. $e$ is even, whereas $x^q y^2$ is always in $(x^{q+1},xy^q)$. Since the completion does not affect the length, if $B_q$ is invertible for every $q$, we get
\[
\ell_{\widehat{R}(\widehat{I})_{a,b}}\left(\frac{\widehat{R}(\widehat{I})_{a,b}}{ \n^{[q]}}\right)=
\begin{cases}
6q-6 \ \ \ \text{ if } e \text{ is odd} \\
6q-5 \ \ \ \text{ if } e \text{ is even},
\end{cases}
\]
where $\n$ is the maximal ideal of $\widehat{R}(\widehat{I})_{a,b}$.
For instance, this occurs when $a \in k$ and $b=0$, as in the case of the amalgamated duplication, or when $a=b=-1$ and $p \neq 5$, see Remark \ref{Particular values}.
\end{ex}

It is also possible to have the opposite behaviour: the second coefficient of the Hilbert-Kunz function of $R$ does not exist, but the one of $R(I)_{a,b}$ exists for every $a,b \in R$ and a suitable ideal $I$.

\begin{ex} \rm \label{last example}
Let $k=\mathbb{Z}/3\mathbb{Z}$ and $R=k[x,y]/(x^4+x^3y+x^2y^2+xy^3+y^4)$. It is easy to see that the reduced Gr\"obner basis ${\rm G}(J)$ of $J=(x^4+x^3y+x^2y^2+xy^3+y^4,x^q,y^q)$ with respect to the lexicographic order $x>y$ depends on the congruence of $q$ modulo $5$. More precisely, if we set $f=x^4+x^3y+x^2y^2+xy^3+y^4$, we have
\[
{\rm G}(J)=\begin{cases}
\{y^q, f, x y^{q-1}\}   &\text{ \ if } q \equiv 1 \mod 5 \\
\{y^q, f, x^2y^{q-2}\} &\text{ \ if } q \equiv 2 \mod 5 \\
\{y^q, f, x^3y^{q-3}, x^2 y^{q-1}\} &\text{ \ if } q \equiv 3 \mod 5 \\
\{y^q, f, x^3y^{q-3}+x^2 y^{q-2}+xy^{q-1}\}\ \ \ \ \ \ &\text{ \ if } q \equiv 4 \mod 5. \\
\end{cases}
\]
It is a straightforward calculation to find the initial ideal of $J$ and the length of $k[x,y]/J$, therefore, since the completion does not affect this length, we get
\[
\ell_{\widehat{R}}\left(\frac{\widehat{R}}{\widehat{\m}^{[q]}\widehat{R}} \right)=
\begin{cases}
4q-4 \ \ \ \ \ \ \text{ if } e \text{ is odd} \\
4q-3 \ \ \ \ \ \ \text{ if } e \text{ is even},
\end{cases}
\]
where $\m=(x,y)$.
On the other hand, if we consider the ideal $I=(x^3,y^3,f)$ of $R$, we can find the Gr\"obner basis of $\m^{[q]}I+(f)$ seen as ideal of $k[x,y]$ obtaining
\[
{\rm G}(\m^{[q]}I+(f))=\begin{cases}
\{y^{q+3}, f, x y^{q+2},x^2 y^{q+1}, x^3 y^{q}  \}   &\text{ \ if } q \equiv 1,4 \mod 5 \\
\{y^{q+3}, f, x^2y^{q+1}, x^3 y^{q}\} &\text{ \ if } q \equiv 2 \mod 5 \\
\{y^{q+3}, f, x y^{q+2}, x^3y^{q}\} &\text{ \ if } q \equiv 3 \mod 5. \\
\end{cases}
\]
This, together with the equalities $\ell_{R}(I/\m^{[q]}I)=\ell_{R}(R/\m^{[q]}I)-\ell_{R}(R/I)$ and $\ell_{R}(R/I)=8$, easily implies that
\[
\ell_{R}\left(\frac{I}{\m^{[q]}I} \right)=
\begin{cases}
4q-1 \ \ \ \ \ \ \text{ if } e \text{ is  odd} \\
4q-2 \ \ \ \ \ \ \text{ if } e \text{ is  even}.
\end{cases}
\]
Hence, since $I=\m^{[p]}$, Corollary \ref{Equality holds} implies that the Hilbert-Kunz function of $R(I)_{a,b}$ is equal to $8q-5$ for every $a$ and $b$, in particular the second coefficient exists.
\end{ex}

Since the polynomial $f$ of the previous example is homogeneous and irreducible in $k[x,y]$, it is also irreducible as an element of $k[[x,y]]$. Indeed, if $f=f_1 f_2$ in $k[[x,y]]$, then the product of the terms with smallest degree of $f_1$ and $f_2$ has to be equal to $f$. Therefore, in the previous example $\widehat{R}$ is a local domain.

The second coefficient of the Hilbert-Kunz function has been systematically studied for the first time in \cite{HMM}, where it is proved its existence for excellent local normal domains with a perfect residue field. Actually, the condition {\it normal domain} can be replaced by {\it ring regular in codimension one}, see \cite{CK,HY}. It is well known that it is not possible to drop this condition, see e.g. the previous examples or the one in  \cite{M}, and our last example shows that it is not possible to drop the regularity in codimension one, even if $R$ is assumed to be an integral domain. We also notice that, starting from such a domain, it is very easy to construct other domains with the same property using the rings $R(I)_{a,b}$, in contrast of the idealization and the amalgamated duplication which are never domains. For instance, if in Example \ref{last example} we consider $a,b \in \widehat{R}$ such that $t^2+at+b$ is irreducible over the total ring of fractions of $\widehat{R}$, we get that $\widehat{R}(I)_{a,b}$ is a domain for every ideal $I$ of $\widehat{R}$ by \cite[Remark 1.10]{BDS}. As explicit examples, the rings $\widehat{R}(\widehat \m)_{0,-1}$ and $\widehat{R}(\widehat \m)_{-1,-1}$ are always domains and in this case $B_q$ is always invertible by Remark \ref{Particular values}. Moreover, Theorem \ref{main} implies that for every $q=p^e$ with $e \in \mathbb{N}$ their Hilbert-Kunz function is equal to
\begin{equation*}
\ell_R\left(\frac{R}{\m^{[q]}}\right)+\ell_R\left(\frac{\m}{\m^{[q]}\m}\right) -\ell_R\left(\frac{\m^{[q]}}{\m^{[q]}\m}\right)=\ell_R\left(\frac{R}{\m^{[q]}}\right)+\ell_R\left(\frac{\m}{\m^{[q]}}\right)=2\ell_R\left(\frac{R}{\m^{[q]}}\right)-1.
\end{equation*}
Therefore, it is equal to
\begin{equation*}
\begin{cases}
8q-9 \ \ \ \ \ \ &\text{ \ if } e \text{ is odd} \\
8q-7 \ \ \ \ \ \ &\text{ \ if } e \text{ is even.}
\end{cases}
\end{equation*}
If $R_1=\widehat{R}(\widehat \m)_{0,-1}$ and $\m_1$ is the unique maximal ideal of $R_1$, also the second coefficient of the Hilbert-Kunz function of $R_2=R_1(\m_1)_{0,-1}$ is periodic and, continuing in this way, it is possible to construct infinitely many domains whose second coefficient is periodic.

\medskip

\noindent
{\bf Acknowledgements}.
Most of this paper was carried out while the first author stayed at the Institute of Mathematics of the University of Barcelona (IMUB). He would like to thank the IMUB for the great hospitality.
The authors also thank Alessio Caminata for some useful discussions about the topics of this paper and the anonymous referee for pointing out some inaccuracies in a preliminary version of the manuscript and the alternative proof of (3) in Theorem \ref{main}.

\end{document}